\documentclass[10pt]{article}
\usepackage{times}
\usepackage{amsmath,amsfonts,amstext,amssymb,amsbsy,amsopn,amsthm,eucal,slashed,amstext,enumerate,accents}

\usepackage{color}
\usepackage{txfonts}
\usepackage{dsfont}
\usepackage{hyperref}
\usepackage{graphicx}   

\numberwithin{equation}{section}
  \theoremstyle{plain}
 \newtheorem{theorem}{Theorem}
 
\newtheorem{proposition}{Proposition}
 \newtheorem{lemma}{Lemma}

 \theoremstyle{remark}

\theoremstyle{definition}

\newcommand{\tr}{\text{tr}}

\newcommand{\ddbar}{\sqrt{-1}\partial\bar\partial}

\title{$L^\infty$ estimates for K\"ahler-Ricci flow on K\"ahler-Einstein Fano manifolds: a new derivation}
\author{Wangjian Jian and Yalong Shi}

\begin{document}
\maketitle

\begin{abstract}
Assuming Perelman's estimates, we give a new proof of uniform $L^\infty$ estimate along normalized K\"ahler-Ricci flow on Fano manifolds with K\"ahler-Einstein metrics, using Chen-Cheng's auxiliary Monge-Amp\`ere equation and the Alexandrov-Bakelman-Pucci maximum principle. This proof does not use pluripotential theory.
\end{abstract}
\tableofcontents

\section{Introduction}
Let $(X^n,\omega_g)$ be a compact K\"ahler manifold of complex dimension $n$. Assume $c_1(X)>0$, i.e. $X$ is Fano and assume $[\omega_g]=2\pi c_1(X)$. We consider the following normalized K\"ahler-Ricci flow: 
\begin{equation}\label{eqn:nKRF}
    \frac{\partial}{\partial t}\omega(t)=\omega(t)-Ric(\omega(t)).
\end{equation}
Choose a smooth volume form $\Omega$ satisfying $\omega_0=-\ddbar\log\Omega$ and $\int_X\Omega=\int_X\omega_0^n$, and set $\omega(t)=\omega_0+\ddbar\varphi$, then \eqref{eqn:nKRF} is equivalent to 
\begin{equation}\label{eqn:potentialKRF}
    \left\{
    \begin{aligned}
   &\frac{\partial\varphi}{\partial t}\quad= \log\frac{(\omega_0+\ddbar \varphi)^n}{\Omega}+\varphi;\\
&\ \varphi\big|_{t=0}\ =0. 
\end{aligned}
\right.
\end{equation}
Set $u_0:=\frac{\partial\varphi}{\partial t}$. Then we find a smooth function $b(t): [0, \infty) \to \mathbb{R}$ such that
\begin{equation}\label{nov}
\int_X e^{- u_0 - b} \omega(t)^n = (2\pi)^{n},
\end{equation}
then we set $u:=u_0+b$. Perelman proved that (cf. \cite{SeT}) we can find a uniform constant $C>0$ such that
\begin{equation}\label{eqn:Perelman_estimate}
  |u|+|\nabla u|_g+|\Delta_g u|\leq C.  
\end{equation}
Also, the diameter and scalar curvature are uniformly bounded.
Based on these estimates, one can prove that if there is a K\"ahler-Einstein metric or more general K\"ahler-Ricci soliton metric, then the normalized K\"ahler-Ricci flow with converge to the corresponding K\"ahler-Einstein or K\"ahler-Ricci soliton metric. For the proof, see Tian-Zhu \cite{TZh}, \cite{TZh2}, Tian-Zhang-Zhang-Zhu \cite{TZZZ} and Collins-Sz\'ekelyhidi \cite{CoSz}. In 2020, B.Guo, D.H.Phong and J.Sturm \cite{GPhSt} provide a shorter proof of the convergence assuming Perelman's estimates. In their work, Kolodziej's results on uniform H\"older estimates plays an important role. In this note, we use Chen-Cheng's idea \cite{ChCh} of using an auxilliary Monge-Ampere equation and the Alexandrov-Bakelman-Pucci maximum principle to give another derivation of the uniform $L^\infty$ estimate along normalized K\"ahler-Ricci flow. 

\begin{theorem}
Let $(X^n,\omega_g)$ be a Fano manifold of complex dimension $n$ and $[\omega_g]=2\pi c_1(X)$. Assume the K-energy is proper (this is the case if $X$ has no nontrivial holomorphic vector fields and admits a K\"ahler-Einstein metric), let $\varphi(t)$ be solutions to \eqref{eqn:potentialKRF}, then we can find a uniform constant $C$ such that $osc\ \varphi(t)\leq C$.
\end{theorem}

Given this oscillation bound, it is not difficult to prove the convergence. For reader's convenience, we outline the argument in the end of the paper.

Note that recently B.Guo, D.H.Phong and F.Tong, using ideas of \cite{ChCh}\cite{WWZh}, have developed a systematic way of using auxiliary Monge-Amp\`ere equations to get uniform bounds for a large class of fully nonlinear elliptic equations, see \cite{GPhT} and \cite{GPh}. X.Chen and J. Cheng \cite{ChCh-arx} derived corresponding $L^\infty$ estimates for parabolic Monge-Amp\`ere equations and Hessian equations.\\

In the next section, we shall derive the $L^\infty$ bound assuming an entropy bound. Then in section 3 we prove the entropy bound assuming the existence of K\"ahler-Einstein metrics. This method could be adopted to more general K\"ahler-Ricci soliton case. We leave the details to the readers.\\

{\bf Acknowledgement}: This work arises from a joint project of the authors with Prof. Jian Song of Rutgers University. The authors would like to thank him for suggesting the problem and for helpful discussions. W. Jian is supported by NSFC No.12201610, NSFC No.12288201, National Key R$\&$D Program of China (Grant No.2021YFA1003100).

\section{Application of ABP assuming entropy bounds}

Let $F:=u_0-\varphi$, then \eqref{eqn:potentialKRF} is equivalent to the following system
\begin{equation}\label{eqn:system}
    \left\{
    \begin{aligned}
       (\omega_0+\ddbar\varphi)^n =e^F\Omega;\\
       \Delta_{g(t)}F =-R(t)+tr_{\omega(t)}\omega_0,
    \end{aligned}
    \right.
\end{equation}
where $R(t)$ is the scalar curvature of $g(t)$, which is uniformly bounded according to Perelman's estimates \cite{SeT}. In this section, we shall prove:

\begin{proposition}\label{prop:L2}
For any $q>1$, there is a uniform constant $C>0$, depending only on $n, q, \omega_0$ and $\int_X\log\frac{\omega_\varphi^n}{\Omega}\omega_{\varphi}^n$, such that
$$\|e^F\|_{L^q(X,\Omega)}\leq C.$$
\end{proposition}

Given this estimate, we can then follow Blocki \cite{Blo} (see also \cite{PhSoSt} page 334-335 for a simple proof when $q>2$) to  get uniform $L^\infty$ bound of $\text{osc}\ \varphi$ using ABP maximum principle.

To prove Proposition \ref{prop:L2}, following Chen-Cheng \cite{ChCh}, we introduce an intermediate function $\psi$ defined by the following equation:
\begin{equation*}
    \left\{
    \begin{aligned}
       (\omega_0+\ddbar\psi)^n &=\frac{e^F\sqrt{1+F^2}}{\int_X e^F\sqrt{1+F^2}\Omega }\Omega\\
       \sup_X \psi &=0.
    \end{aligned}
    \right.
\end{equation*}
Note that since $F$ depends on $t$, so does $\psi$. By Yau's theorem, $\psi\in C^\infty(X)$.

\begin{lemma}\label{lem:F upper}
For any $q>1$, there is a constant $C>0$, depending on the ``entropy'' $\int_X \log\frac{\omega_\varphi^n}{\Omega}\omega_\varphi^n=\int_X Fe^F\Omega$, and a constant $0<\epsilon<<1$ such that
\begin{equation}
    F+\epsilon\psi-q\varphi\leq C.
\end{equation}
\end{lemma}

\begin{proof}
We fix a $t>0$ and all the discussions below are at time $t$.

Consider the function $Q:=e^{\delta(F+\epsilon\psi-q\varphi)}$, where $0<\delta<<1$ is a constant, to be chosen later. Assume $Q$ achieves its maximum (on $X$) at $p_0\in X$. 

We shall use the following ``cut-off'' function. Pick $\theta\in (0,1)$, to be determined later, and choose $d_0>0$, which is comparable to the injectivity radius of $g_0$, then we can find a smooth function $f$ satisfying: 
\begin{itemize}
    \item $1-\theta\leq f\leq 1$;
    \item $f(p_0)=1$;
    \item $f\equiv 1-\theta$ on $X\setminus B_{\frac{d_0}{2}}(p_0)$, where $B_{\frac{d_0}{2}}(p_0)$ is a geodesic ball with respect to $g_0$;
    \item $|\nabla f|^2_{g_0}\leq \frac{4\theta^2}{d_0^2}$;
    \item $|\nabla^2 f|_{g_0}\leq\frac{4\theta}{d_0^2}$.
\end{itemize}
Then we have
\begin{align*}
    \Delta_{\omega_\varphi}(Qf)&= f\Delta_{\omega_\varphi} Q+2\langle \nabla f,\nabla Q\rangle_{\omega_\varphi} +Q\Delta_{\omega_\varphi}f\\
    &=\delta fe^{\delta(F+\epsilon\psi-q\varphi)}\Delta_{\omega_\varphi} \big(F+\epsilon\psi-q\varphi\big)+e^{\delta(F+\epsilon\psi-q\varphi)}\Delta_{\omega_\varphi}f\\
    &\quad+\delta^2 fe^{\delta(F+\epsilon\psi-q\varphi)}|\nabla(F+\epsilon\psi-q\varphi)|_{\omega_\varphi}^2\\
    &\quad+2\delta e^{\delta(F+\epsilon\psi-q\varphi)}\langle \nabla f,\nabla (F+\epsilon\psi-q\varphi)\rangle_{\omega_\varphi}.
\end{align*}
For the first term, we have
\begin{align*}
\Delta_{\omega_\varphi} \big(F+\epsilon\psi-q\varphi\big)&=-R+tr_{\omega_\varphi}\omega_0-qn+qtr_{\omega_\varphi}\omega_0+\epsilon \Delta_{\omega_\varphi}\psi \\
&\geq -R-qn +(q+1-\epsilon)tr_{\omega_\varphi}\omega_0+\epsilon n\Big(\frac{\omega^n_\psi}{\omega^n_\varphi}\Big)^{\frac{1}{n}}\\
&=-R-qn +(q+1-\epsilon)tr_{\omega_\varphi}\omega_0+\epsilon n\frac{(F^2+1)^{\frac{1}{2n}}}{\big(\int_X e^F\sqrt{1+F^2}\Omega \big)^{\frac{1}{n}}}.
\end{align*}
For the second term, we have
\begin{align*}
 e^{\delta(F+\epsilon\psi-q\varphi)}\Delta_{\omega_\varphi}f &\geq -e^{\delta(F+\epsilon\psi-q\varphi)}|\nabla^2 f|_{g_0} tr_{\omega_\varphi}\omega_0 \\
 &\geq -e^{\delta(F+\epsilon\psi-q\varphi)}\frac{4\theta}{d_0^2}tr_{\omega_\varphi}\omega_0 
 \end{align*}
On the other hand, we have
\begin{align*}
  2\delta \langle \nabla f,\nabla (F+\epsilon\psi-q\varphi)\rangle_{\omega_\varphi} &\geq -\delta^2 f|\nabla(F+\epsilon\psi-q\varphi)|_{\omega_\varphi}^2 -\frac{|\nabla f|_{\omega_\varphi}^2}{f}\\
  &\geq -\delta^2 f|\nabla(F+\epsilon\psi-q\varphi)|_{\omega_\varphi}^2 -\frac{|\nabla f|_{g_0}^2}{1-\theta}tr_{\omega_\varphi}\omega_0\\
  &\geq -\delta^2 f|\nabla(F+\epsilon\psi-q\varphi)|_{\omega_\varphi}^2 -\frac{4\theta^2 }{d_0^2(1-\theta)}tr_{\omega_\varphi}\omega_0.
\end{align*}
Adding together, we get
\begin{align*}
    \Delta_{\omega_\varphi}(Qf)&\geq \delta fe^{\delta(F+\epsilon\psi-q\varphi)} \Big(-R-qn+\epsilon n\frac{(F^2+1)^{\frac{1}{2n}}}{\big(\int_X e^F\sqrt{1+F^2}\Omega \big)^{\frac{1}{n}}}\Big)\\
    &\quad +e^{\delta(F+\epsilon\psi-q\varphi)}\Big(\delta f(q+1-\epsilon)-\frac{4\theta}{d_0^2}-\frac{4\theta^2 }{d_0^2(1-\theta)}\Big)tr_{\omega_\varphi}\omega_0.
\end{align*}
When $\theta$ is chosen to be small enough (depending on $d_0,\delta,\epsilon$ only), we get
\begin{equation}\label{eqn:ineq}
  \Delta_{\omega_\varphi}\Big(e^{\delta(F+\epsilon\psi-q\varphi)}f\Big)\geq  \delta fe^{\delta(F+\epsilon\psi-q\varphi)} \Big(-R-qn+\epsilon n\frac{(F^2+1)^{\frac{1}{2n}}}{\big(\int_X e^F\sqrt{1+F^2}\Omega \big)^{\frac{1}{n}}}\Big). 
\end{equation}
Now we apply the ABP maximum principle to $Qf$ on the geodesic ball $B_{d_0}(p_0)$ (with respect to $g_0$):
\begin{align*}
    \max_X Q &=\max_X(Qf)\\
    &\leq \sup_{\partial B_{d_0}(p_0)}(Qf)+C(n,\omega_0) d_0\left\{\int_{B_{d_0}(p_0)}\frac{(\delta Qf)^{2n}\Big(-R-qn+\epsilon n\frac{(F^2+1)^{\frac{1}{2n}}}{\big(\int_X e^F\sqrt{1+F^2}\Omega \big)^{\frac{1}{n}}}\Big)_{-}^{2n}}{e^{-2F}}\omega_0^n\right\}^{\frac{1}{2n}}.
\end{align*}
Notice that when $$-R-qn+\epsilon n\frac{(F^2+1)^{\frac{1}{2n}}}{\big(\int_X e^F\sqrt{1+F^2}\Omega \big)^{\frac{1}{n}}}\leq 0,$$ we will get
\begin{align*}
|F|&<\sqrt{F^2+1}\leq C(q,n,\epsilon,\omega_0)\int_X e^F\sqrt{F^2+1}\Omega\\
&\leq C(q,n,\epsilon,\omega_0)\Big([\omega_0]^n+\int_X |F|e^F\Omega\Big)\\
&\leq C(q,n,\epsilon,\omega_0)\Big([\omega_0]^n+\int_X Fe^F\Omega+2\int_{\{F<0\}}(-F)e^F\omega_0^n\Big)\\
&\leq C(q,n,\epsilon,\omega_0)\Big(2[\omega_0]^n+\int_X Fe^F\Omega\Big).
\end{align*}
So we obtain
\begin{align*}
    \max_X Q &\leq (1-\theta)\max_X Q +C\delta d_0\Big[\int_X e^{2n\delta(\epsilon\psi-q\varphi)}\omega_0^n\Big]^{\frac{1}{2n}}\\
    &\leq (1-\theta)\max_X Q +C\delta d_0\Big[\int_X e^{-2qn\delta\varphi}\omega_0^n\Big]^{\frac{1}{2n}}.
\end{align*}
Using Tian's $\alpha$-invariant, take $\delta:=\frac{1}{4qn}\alpha([\omega_0])$, we get uniform upper bound of $Q$ and hence of $F+\epsilon\psi-q\varphi$.
\end{proof}

\begin{proof}[Proof of Proposition \ref{prop:L2}]
From $F+\epsilon\psi-q\varphi\leq C$ we get
$$-(q+1)\varphi\leq -u-\epsilon\psi+C\leq -\epsilon\psi+C'.$$
Then we have
\begin{align*}
    \int_X e^{qF}\omega_0^n &\leq C\int_X e^{-q\varphi}\omega_0^n\\
    &\leq \tilde C\int_X e^{-\frac{q\epsilon}{q+1}\psi}\omega_0^n\leq\tilde{\tilde C},
\end{align*}
when we choose $\epsilon\leq \alpha([\omega_0])$.
\end{proof}

\section{Entropy bound, $L^\infty$ estimate and convergence}

Now we assume that $X$ admits a K\"ahler-Einstein metric and does not have any non-trivial holomorphic vector fields. Then it is well-known that the K-energy is proper on the space of K\"ahler potentials, i.e. there are constants $\mu>0, C>0$ such that
\begin{equation}\label{eqn:proper}
    K_{\omega_0}(\psi)\geq \mu (I_{\omega_0}-J_{\omega_0})(\psi)-C,\quad \forall \psi\in \mathcal{H}(X,\omega_0).
\end{equation}

\begin{lemma}
Under the above assumption, along the normalized K\"ahler-Ricci flow, we have
$$\int_X Fe^F\Omega\leq C.$$
\end{lemma}

\begin{proof}

Recall that
\begin{align*}
    \frac{d}{dt}K_{\omega_0}(\varphi)&=-\int_X \frac{\partial\varphi}{\partial t}(Ric(\omega_\varphi)-\omega_\varphi)\wedge \frac{\omega_\varphi^{n-1}}{(n-1)!}\\
    &=\int_X \frac{\partial\varphi}{\partial t}\ddbar\Big(\varphi+\log\frac{\omega_\varphi^n}{\Omega}\Big)\wedge \frac{\omega_\varphi^{n-1}}{(n-1)!}\\
    &= \int_X \frac{\partial\varphi}{\partial t}\ddbar\frac{\partial\varphi}{\partial t}\wedge \frac{\omega_\varphi^{n-1}}{(n-1)!}\\
    &=-\int_X\Big|\frac{\partial\varphi}{\partial t}\Big|^2\frac{\omega_\varphi^{n}}{n!}\leq 0.
\end{align*}
So in this case $K_{\omega_0}(\varphi)$ is uniformly bounded along the normalized K\"ahler-Ricci flow. On the other hand, we have 
\begin{align*}
    K_{\omega_0}(\varphi)&=\int_0^1\int_X \varphi\Big(\omega_{\tau\varphi}-Ric(\omega_{\tau\varphi})\Big)\wedge\frac{\omega_{\tau\varphi}^{n-1}}{(n-1)!}d\tau\\
    &=\int_0^1\int_X \Big(\tau\varphi+\log\frac{\omega_{\tau\varphi}^n}{\Omega}\Big)\ddbar\varphi\wedge\frac{\omega_{\tau\varphi}^{n-1}}{(n-1)!}d\tau\\
    &= \int_X \Big(\varphi+\log\frac{\omega_{\varphi}^n}{\Omega}\Big)\frac{\omega_{\varphi}^n}{n!}-\int_X \log\frac{\omega_0^n}{\Omega}\frac{\omega_0^n}{n!}-\int_0^1\int_X \varphi\frac{\omega_{\tau\varphi}^n}{n!} d\tau \\
    &=\int_X Fe^F\frac{\Omega}{n!}-C+\int_0^1\int_X \tau\varphi\ddbar\varphi\wedge\frac{\omega_{\tau\varphi}^{n-1}}{(n-1)!}d\tau\\
    &=\int_X Fe^F\frac{\Omega}{n!}-C+\int_0^1\int_X \varphi(\omega_{\tau\varphi}-\omega_0)\wedge\frac{\omega_{\tau\varphi}^{n-1}}{(n-1)!}d\tau\\
    &=\int_X Fe^F\frac{\Omega}{n!}-C-J_{\omega_0,\omega_0}(\varphi)\\
    &=\int_X Fe^F\frac{\Omega}{n!}-C-(I_{\omega_0}-J_{\omega_0})(\varphi).
\end{align*}
So we get
$$\int_X Fe^F\frac{\Omega}{n!}\leq C+K_{\omega_0}(\varphi)+(I_{\omega_0}-J_{\omega_0})(\varphi)\leq (1+\frac{1}{\mu})K_{\omega_0}(\varphi)+C'\leq C''.$$
\end{proof}

Combining with Proposition \ref{prop:L2}, we get uniform estimates of $osc\ \varphi(t)$.\\

It is well known that given this oscillation bound, we can bound $|\varphi(t)|_{L^\infty}$ after suitable normalization, see, for example \cite{PSSW}. Indeed, if we set 
$\alpha(t):=\frac{1}{V}\int_X \dot\varphi \omega(t)^n$, then we have
$$\dot\alpha=\alpha-\frac{1}{V}\int_X |\nabla\dot\varphi|^2 \omega(t)^n=\alpha-\frac{1}{V}\int_X |\nabla u|^2 \omega(t)^n.$$
Following Chen-Tian\cite{ChT} and Phong-Song-Sturm-Weinkove \cite{PSSW}, by adding a constant to $\varphi(0)=0$, we can make $\alpha$ uniformly bounded. In fact, we have
$$\alpha(t)=e^t\Big(\alpha(0)-\int_0^t e^{-s}\big[\fint_X |\nabla u|^2(s) \big]ds\Big).$$
In general, $\alpha(t)$ may not be uniformly bounded. We need to replace the initial condition in \eqref{eqn:potentialKRF} to $\varphi(0)=c_0$ for a constant $c_0$, determined in the following way:
Since $|\nabla u|$ is bounded by Perelman, $\int_0^\infty e^{-s}\big[\fint_X |\nabla u|^2(s) \big]ds$ converges. To make $\alpha(t)$ uniformly bounded, we need to ensure $$\alpha(0)=\int_0^\infty e^{-s}\big[\fint_X |\nabla u|^2(s) \big]ds.$$
Note that by \eqref{eqn:potentialKRF}, we have
$$\alpha(0)=\fint_X \dot\varphi|_{t=0}=\fint_X\log\frac{\omega_0^n}{\Omega}+\fint \varphi(0).$$
Then we can set the initial condition to $$\varphi(0)=c_0:=\int_0^\infty e^{-s}\big[\fint_X |\nabla u|^2(s) \big]ds-\fint_X\log\frac{\omega_0^n}{\Omega}.$$
Note that by uniqueness of the solutions to the K\"ahler-Ricci flow, the new $\varphi(t)$ differs from the old one by a constant depending only on $t$.

Under this re-choice of intial potential, since $u$ is uniformly bounded and $\dot\varphi$ differs from $u$ by a time-dependent constant, we have a uniform bound for $osc\ \dot\varphi$. Since the integration average of $\dot\varphi$ is uniformly bounded, we conclude that $\dot\varphi$ is also uniformly bounded. However, we have $F=\dot\varphi-\varphi$ and $\int e^F\Omega=const$, the oscillation bound $osc\ \varphi\leq C$ is equivalent to $L^\infty$ bound for $\varphi$. Now both $\dot\varphi$ and $\varphi$ are uniformly bounded, so is $F$.\\

Now we can bound $\tr_{\omega_0}\omega(t)$ using the parabolic maximum principle as follows: standard computation gives us
$$(\partial_t-\Delta_{\omega(t)})\log tr_{\omega_0}\omega(t)\leq 1+C tr_{\omega(t)}\omega_0,$$
where $C$ depends only on the bisectional curvature of $\omega_0$. Since $(\partial_t-\Delta_{\omega(t)})F=n-tr_{\omega(t)}\omega_0$, we have
$$(\partial_t-\Delta_{\omega(t)})[\log tr_{\omega_0}\omega(t)+(C+1)F]\leq 1+n(C+1)- tr_{\omega(t)}\omega_0.$$
For any $T>0$, if the maximum of $Q:=\log tr_{\omega_0}\omega(t)+(C+1)F$ is achieved at $(p,t)$, write the eigenvalues of $\omega(t)$ with respect to $\omega_0$ as $\lambda_1,\dots,\lambda_n$, then at $(p,t)$, we have
$$\frac{1}{\lambda_{min}}<tr_{\omega(t)}\omega_0\leq 1+n(C+1)=:C_1.$$
Now by the equation and bounds for $F$, we have $\lambda_1\dots\lambda_n=\frac{\omega(t)^n}{\omega_0^n}\leq C_2$, this implies
$$\lambda_{max}\leq \frac{C_2}{\lambda_{min}^{n-1}}\leq C_1^{n-1}C_2=:C_3.$$
From this, we get a uniform bound for $Q$ independent of $T$ and hence a uniform bound for $tr_{\omega_0}\omega(t)$, which in turn implies uniform equivalence of $\omega(t)$ and $\omega_0$.

Once we have uniform $C^2$ estimates, one can prove the higher order estimates and convergence of the flow by standard arguments in parabolic theory, see, for example \cite{ChT} and \cite{PSSW}.

\bibliographystyle{alpha}

\vspace{.5cm}

\noindent Wangjian Jian, Institute of Mathematics, Academy of Mathematics and Systems Science, Chinese Academy of Sciences, Beijing, 100190, China.\\
{\em Email: wangjian@amss.ac.cn}\\

\noindent Yalong Shi, Department of Mathematics, Nanjing University, Nanjing, 210093, China.\\
{\em Email: shiyl@nju.edu.cn}

\end{document}